\documentclass[12pt]{article}

\usepackage{amsmath,amssymb,amsthm}

\newtheorem{theorem}{Theorem}
\newtheorem{lemma}{Lemma}
\newtheorem{corollary}{Corollary}

\newtheorem{claim}{Claim}

\newcommand{\per}{\operatorname{per}}

\title{On the Number of Hamiltonian Cycles in a Boolean Cube}
\author{A.\,L.\,Perezhogin \and V.\,N.\,Potapov}
\date{}

\begin{document}

\maketitle

\begin{abstract}
It is shown that, as $n\to\infty$, the logarithm of the number of
decompositions into cycles of the $n$-dimensional Boolean cube $E^n$
is
\[
2^n(\ln n-1+o(1)),
\]
and the logarithm of the number of Hamiltonian cycles in $E^n$ is at least
\[
2^{n-1}(\ln n-1+o(1)).
\]
It is proved that, in $E^n$, every perfect matching whose edges belong
to at most $k$ directions can be extended to a Hamiltonian cycle for
every $n\geq n_0(k)$.
\end{abstract}

\section*{Introduction}

Let $E^n$ denote the $n$-dimensional Boolean cube, i.e., the graph whose
vertices are binary vectors of length $n$, where the vertices
$x=(\alpha_1,\ldots,\alpha_n)$ and
$y=(\beta_1,\ldots,\beta_n)$ are adjacent if and only if their Hamming
distance
\[
\rho(x,y)=\sum_{i=1}^n(\alpha_i\mathbin{\oplus}\beta_i)=1,
\]
where $\mathbin{\oplus}$ denotes addition modulo $2$.

The problem of determining the number $h_n$ of Hamiltonian cycles in
$E^n$ was posed in \cite{Gil} and considered in
\cite{Abb,Di,Do,Mol}.  The best known lower bound for $h_n$ as
$n\to\infty$ (see \cite{Mol}) is
\[
\ln h_n \geq 2^{n-1}(c+o(1)),
\qquad c\approx 1.42.
\]
The best known upper bound for $h_n$ was obtained in \cite{Di}:
\[
\ln h_n\leq
2^n\left(\ln n-\frac{\ln 2}{2}+o(1)\right).
\]

In the present paper, we establish the following asymptotic bounds
for $h_n$ as $n\to\infty$:
\[
2^{n-1}(\ln n-1+o(1))
\leq
\ln h_n
\leq
2^n(\ln n-1+o(1)).
\]

In obtaining these bounds, we use the asymptotic formulas, proved
below, for the logarithm of the number $m_n$ of perfect matchings and
the logarithm of the number $f_n$ of $2$-factors (decompositions into
simple cycles) in $E^n$:
\[
\ln m_n=2^{n-1}(\ln n-1+o(1)),
\]
and
\begin{equation}
\ln f_n=2^n(\ln n-1+o(1)).
\label{1}
\end{equation}
These asymptotic formulas are obtained from known results on permanents
of square $(0,1)$-matrices; see \cite{Bregman,Minc,Schrijver}.

Notice that the known upper bounds for $h_n$ follow from upper bounds
for $f_n$ and the inequality
\[
h_n\leq f_n.
\]
This method of obtaining upper bounds for $h_n$ appears to be exhausted,
since estimate~\eqref{1} is asymptotically exact.

In \cite{Kreweras}, the conjecture was stated that every perfect
matching in $E^n$ can be extended to a Hamiltonian cycle. It was shown
there that this assertion is true in $E^n$ for $n\leq4$.
Below, we prove that every perfect matching in $E^n$ whose edges belong
to at most $k$ directions can be extended to a Hamiltonian cycle for
every
\[
n\geq n_0(k).
\]

\section{Preliminaries}

We introduce some notation. Let $G$ be an arbitrary graph.
We denote by $V(G)$ the set of vertices and by $E(G)$ the set of
edges of $G$. By $k(G)$ we denote the number of connected components
of $G$.

An arbitrary set of pairwise nonadjacent edges of $G$ is called a
\emph{matching} in $G$. A matching $A\subset E(G)$ is called
\emph{perfect} if every vertex of $G$ is incident with exactly one
edge of $A$. An arbitrary collection of pairwise vertex-disjoint
simple cycles in $G$ that covers all vertices of $G$ is called a
\emph{2-factor} of $G$. A \emph{Hamiltonian cycle} in $G$ is a
2-factor consisting of a single cycle. Clearly, two disjoint perfect
matchings in $G$ generate a 2-factor.

Let $G$ be a bipartite graph with $2N$ vertices, and suppose that the
vertices in each part are labeled by the numbers
$1,2,\ldots,N$. The \emph{adjacency matrix} of $G$ is the square
$(0,1)$-matrix
\[
A(G)=(a_{ij})
\]
of order $N$, where $a_{ij}=1$ if the $i$th vertex of the first part
is adjacent to the $j$th vertex of the second part, and $a_{ij}=0$
otherwise.

The \emph{permanent} of the matrix $A(G)$ is defined by
\[
\per A(G)
=
\sum_{\tau\in S_N}\prod_{i=1}^N a_{i\tau(i)},
\]
where $S_N$ is the set of all permutations of $\{1,\ldots,N\}$.
It is known (see, for example, \cite{Minc}) that the number of
perfect matchings in a bipartite graph $G$ is equal to the permanent
of its adjacency matrix $A(G)$.

L.~M.~Bregman \cite{Bregman} proved the following upper bound for the
permanent of a square $(0,1)$-matrix $B$ of order $N$, conjectured by
H.~Minc \cite{Minc}:
\[
\per B
\leq
\prod_{i=1}^N (r_i!)^{1/r_i},
\]
where $r_i$ is the sum of the entries in the $i$th row of $B$.
It follows immediately from this inequality that
\begin{equation}
\per A(G)
\leq
(k!)^{\frac{N}{k}}
\label{2}
\end{equation}
for the number of distinct perfect matchings in a $k$-regular
($k$-uniform) bipartite graph $G$ with $2N$ vertices.

A square matrix with nonnegative entries is called
\emph{doubly stochastic} if the sum of the entries in every row and
every column is equal to $1$.

In 1980, G.~P.~Egorychev \cite{Egorychev} and
D.~I.~Falikman \cite{Falikman} proved the van der Waerden conjecture,
which states that the permanent of an arbitrary doubly stochastic
matrix of order $N$ containing at least two distinct entries is
greater than the permanent of the doubly stochastic matrix of order
$N$ whose entries are all equal.

The permanent of the latter matrix is
\[
\frac{N!}{N^N}.
\]
If every entry of the matrix $A(G)$ is divided by $k$, we obtain a
doubly stochastic matrix. Therefore, for $k<N$,
\begin{equation}
\per A(G)
>
(N!)\left(\frac{k}{N}\right)^N.
\label{3}
\end{equation}

In \cite{Schrijver} it was proved that if $G$ is a $k$-regular
bipartite graph with $2N$ vertices and $N\geq k$, then $G$ contains at
least
\begin{equation}
\left(
\frac{(k-1)^{k-1}}{k^{k-2}}
\right)^N
\label{4}
\end{equation}
distinct perfect matchings.

From the refined Stirling formula
\begin{equation}
\left(\frac{N}{e}\right)^N
\sqrt{2\pi N}\,e^{1/(12N+1)}
\leq N!
\leq
\left(\frac{N}{e}\right)^N
\sqrt{2\pi N}\,e^{1/(12N)}
\label{5}
\end{equation}
and the inequality
\[
\left(1-\frac{1}{k}\right)^{k-1}>\frac{1}{e},
\]
it follows that
\[
\lim_{N\to\infty}
\frac{1}{N}
\ln\left(
N!\left(\frac{k}{N}\right)^N
\right)
=
\ln k-1
<
\]
\[
\lim_{N\to\infty}
\frac{1}{N}
\ln\left(
\left(
\frac{(k-1)^{k-1}}{k^{k-2}}
\right)^N
\right)
=
\ln k-\ln\left(1-\frac{1}{k}\right)^{k-1}.
\]

Thus, for fixed $k$ and $N\to\infty$, estimate~\eqref{4} is better
than estimate~\eqref{3}.

\section{Upper and Lower Bounds for the Number of Perfect Matchings
and 2-Factors in $E^n$}

Let $M_n$ denote the set of perfect matchings in $E^n$, and let
$M_n(A)$ denote the set of perfect matchings in $E^n$ that are
disjoint from a perfect matching $A$.

Clearly, $E^n$ is a bipartite $n$-regular graph with $2^n$ vertices.

\begin{theorem}
For $n\geq 2$,
\[
2^{n-1}\left(
\ln n-1\right)
\leq
\ln |M_n|
\leq
2^{n-1}\left(
\ln n-1+
\frac{1}{2n}
\left(\ln(2\pi n)+\frac{1}{6n}\right)
\right).
\]
\end{theorem}

\begin{proof}
The upper bound for $\ln|M_n|$ follows from \eqref{2} and \eqref{5}.
The lower bound for $\ln|M_n|$ follows from \eqref{4} and the inequality
\begin{equation}
\ln\left(1-\frac{1}{k}\right)=-\sum\limits_{i=1}^\infty\frac{1}{ik^i}
\geq
-\frac{1}{k}-\frac{1}{k^2},
\label{6}
\end{equation}
valid for $k\geq 2$.

Notice that the asymptotic estimate
\[
\ln|M_n|
\geq
2^{n-1}(\ln n-1+o(1)),
\qquad n\to\infty,
\]
can also be obtained from \eqref{3}.
\end{proof}

\begin{claim}
Let $A$ be an arbitrary perfect matching in $E^n$. Then, for $n\geq2$,
\[
2^{n-1}\left(
\ln(n-1)-1
\right)
\leq
\ln|M_n(A)|
\]
\[
\leq
2^{n-1}\left(
\ln(n-1)-1+
\frac{1}{2(n-1)}
\left(
\ln(2\pi(n-1))+\frac{1}{6(n-1)}
\right)
\right).
\]
\end{claim}

\begin{proof}
Delete from $E^n$ the edges of the matching $A$. The resulting graph
is a bipartite $(n-1)$-regular graph. Therefore, the estimates for
$\ln|M_n(A)|$ follow from \eqref{2} and \eqref{4}.
\end{proof}

Let $F_n$ denote the set of 2-factors in $E^n$. We denote by
$F_n(i)\subset F_n$ the set of 2-factors, each of which has $i$
connected components.

\begin{claim}
For every $i$, $1\leq i\leq 2^{n-2}$,
\[
|F_n(i)|
\leq
n^{2^n-i}2^{2^n}.
\]
\end{claim}

\begin{proof}
Since the shortest simple cycle in $E^n$ has length $4$, every
2-factor $U\in F_n(i)$ consists of $i$ simple cycles, where
\[
1\leq i\leq 2^{n-2}.
\]
By deleting one edge from each cycle, we obtain a subgraph $U'$
that is a decomposition of the hypercube $E^n$ into $i$ pairwise
vertex-disjoint paths.

Clearly, a 2-factor $U\in F_n(i)$ is uniquely determined by the
subgraph $U'$.

To specify an arbitrary decomposition of $E^n$ into $i$ pairwise
vertex-disjoint simple paths, it is sufficient to choose $i$ vertices
as endpoints of the paths (one endpoint for each path) and to specify
the directions of the edges incident with all remaining vertices of
the Boolean cube $E^n$.

Consequently, the number of distinct decompositions of the Boolean
cube $E^n$ into $i$ pairwise vertex-disjoint paths is at most
\[
n^{2^n-i}\binom{2^n}{i}.
\]
Therefore,
\[
|F_n(i)|
\leq
n^{2^n-i}2^{2^n}.
\]
\end{proof}

\begin{theorem}
For $n\geq2$,
\[
2^n\left(
\ln(n-1)-1-\frac{4}{\ln n}
\right)
\leq
\ln|F_n|
\]
\[
\leq
2^n\left(
\ln n-1+
\frac{1}{2n}
\left(
\ln(2\pi n)+\frac{1}{6n}
\right)
\right).
\]
\end{theorem}

\begin{proof}
\textit{Upper bound.}
Every 2-factor is generated by two disjoint perfect matchings; hence
\[
|F_n|\leq |M_n|^2.
\]
Therefore, the required inequality follows from Theorem~1.

\medskip
\noindent
\textit{Lower bound.}
Each connected component of a 2-factor $U\in F_n(i)$ can be divided
into two vertex-disjoint sets of edges of equal cardinality.
Since each of the two alternating sets of edges of a cycle can be
assigned to either the first or the second perfect matching, every
2-factor $U$ is generated by $2^i$ distinct ordered pairs of perfect
matchings.

The number of ordered pairs of disjoint perfect matchings is at least
$|M_n|m$, where
\[
m=\min_{B\in M_n}|M_n(B)|.
\]
Thus,
\begin{equation}
\sum_{i=1}^{2^{n-2}} |F_n(i)|2^i
\geq
|M_n|m.
\label{7}
\end{equation}

Let
\[
F_n^1=\bigcup_{i=1}^{I}F_n(i),
\qquad
F_n^2=\bigcup_{i=I+1}^{2^{n-2}}F_n(i),
\]
where
\[
I=\left\lfloor\frac{2^{n+2}}{\ln n}\right\rfloor.
\]
Then, by the preceding claim,
\[
|F_n^2|
\leq
2\frac{(2n)^{2^n}}
{n^{2^{n+2}/\ln n}},
\]
that is,
\begin{equation}
\ln|F_n^2|
\leq
\ln2+2^n(\ln n+\ln2-4).
\label{8}
\end{equation}

From \eqref{8}, Theorem~1, and the preceding claim, we obtain
\begin{equation}
|F_n^2|
\leq
\frac{|M_n|m}{2^{1+2^{n-2}}}.
\label{9}
\end{equation}

It follows from \eqref{7} and \eqref{9} that
\[
\sum_{i=1}^{I}|F_n(i)|2^i
\geq
|M_n|m-|F_n^2|2^{2^{n-2}}
\geq
\frac{|M_n|m}{2}.
\]
Consequently,
\[
|F_n|
\geq
|F_n^1|
\geq
\frac{|M_n|m}{2^{I+1}}
\geq
\frac{|M_n|m}
{2^{1+2^{n+2}/\ln n}},
\]
and the lower bound for $\ln|F_n|$ follows from Theorem~1 and the
preceding claim.
\end{proof}

\section{A Lower Bound for the Number of Hamiltonian Cycles in $E^n$}

To establish a lower bound for the number of Hamiltonian cycles in
$E^n$, we shall need some notation.

For an arbitrary set of edges $A$ in $E^n$, let $G(A)$ denote the
subgraph of $E^n$ generated by the set of edges $A$. We say that a
perfect matching $A$ in $E^n$ can be extended to $l$ simple paths
(respectively, to a 2-factor with $l$ connected components) if there
exists a perfect matching $B$ in $E^n$ such that the graph $G(A\cup B)$
consists of $l$ pairwise vertex-disjoint simple paths (respectively,
$G(A\cup B)$ is a 2-factor in the hypercube and
$k(G(A\cup B))=l$). In this case, the graph $G(A\cup B)$ will be called
an extension of the perfect matching $A$ to $l$ simple paths
(respectively, to a 2-factor consisting of $l$ connected components).

For every vector
\[
x=(\alpha_1,\ldots,\alpha_n)\in V(E^n),
\]
define
\[
\varphi_i(x)=(\alpha_1,\ldots,\alpha_n,i),
\qquad i\in\{0,1\}.
\]
For every edge
\[
e=(x_1,x_2)\in E(E^n),
\]
put
\[
\varphi_i(e)=\bigl(\varphi_i(x_1),\varphi_i(x_2)\bigr),
\qquad i\in\{0,1\}.
\]
For an arbitrary set of edges
\[
A=\{e_1,\ldots,e_s\}\subset E(E^n),
\]
put
\[
\varphi_i(A)=
\{\varphi_i(e_1),\ldots,\varphi_i(e_s)\},
\qquad i\in\{0,1\}.
\]

\begin{lemma}
Let $A$ and $B$ be perfect matchings in $E^n$, and suppose that $A$
can be extended to $l$ simple paths. Then the perfect matching
\[
\varphi_0(A)\cup\varphi_1(B)\in M_{n+1}
\]
can be extended to a 2-factor having at most $l$ connected components.
\end{lemma}

\begin{proof}
We construct in $E^{n+1}$ a perfect matching which extends the perfect
matching
\[
\varphi_0(A)\cup\varphi_1(B)
\]
to a 2-factor having at most $l$ connected components.

Let
\[
C=A\cup B
\]
and
\[
G(C)=\Phi_1\cup\cdots\cup\Phi_k,
\]
where $\Phi_1,\ldots,\Phi_k$ are the connected components of $G(C)$.

By the assumption of the lemma, there exists a perfect matching $D$
in $E^n$ such that $G(A\cup D)$ has $l$ connected components.
Since $A\subset C$, the graph $G(C\cup D)$ has $m$ connected components,
where
\[
m\leq l.
\]

Consider the auxiliary graph $\Delta$ whose vertices are
$\Phi_1,\ldots,\Phi_k$. Two vertices $\Phi_i$ and $\Phi_j$, $i\ne j$,
are joined by an edge if there exists an edge
\[
(x_i,x_j)\in D
\]
such that
\[
x_i\in V(\Phi_i),\qquad x_j\in V(\Phi_j).
\]
For each edge
\[
\varepsilon=(\Phi_i,\Phi_j)
\]
of $\Delta$, choose one edge
\[
e(\varepsilon)=(x_i,x_j)
\]
having this property.

Since
\[
V(G(C))=V(E^n),
\]
the graph $\Delta$ has $m$ connected components, just as
$G(C\cup D)$ does.

Let $\Theta$ be a spanning forest of $\Delta$, and put
\[
T=\{e(\varepsilon)\in D:\varepsilon\in E(\Theta)\}.
\]
The forest $\Theta$ consists of $m$ trees on $k$ vertices. Therefore
\[
|T|=|E(\Theta)|=k-m=p.
\]
Order the elements of $E(\Theta)$ arbitrarily:
\[
E(\Theta)=\{\varepsilon_1,\ldots,\varepsilon_p\}.
\]
By the definition of a spanning forest, the graph
\[
\bigl(V(\Delta),\{\varepsilon_1,\ldots,\varepsilon_i\}\bigr)
\]
has $k-i$ connected components for $0\leq i\leq p$. Consequently,
\begin{equation}
k\bigl(G(C\cup\{e_1,\ldots,e_i\})\bigr)=k-i,
\qquad 0\leq i\leq p,
\label{l1n1}
\end{equation}
where
\[
e_i=e(\varepsilon_i).
\]

Let
\begin{equation}
e_i=(x_0^i,x_1^i),
\qquad 1\leq i\leq p.
\label{l1n2}
\end{equation}
Notice that all the vertices
\[
x_0^i,\ x_1^i,\qquad 1\leq i\leq p,
\]
are distinct, since $T$ is contained in the perfect matching $D$.

We now construct a sequence of perfect matchings in $E^{n+1}$:
\[
Y_0=
\{(\varphi_0(x),\varphi_1(x)):x\in V(E^n)\},
\]
and, for $1\leq i\leq p$,
\[
Y_i=
Y_{i-1}\cup\{\varphi_0(e_i),\varphi_1(e_i)\}
\setminus
\left\{
(\varphi_0(x_0^i),\varphi_1(x_0^i)),
(\varphi_0(x_1^i),\varphi_1(x_1^i))
\right\}.
\]

For an arbitrary set of edges $I$ in $E^{n+1}$, define
\[
\pi_n(I)=
\left\{
e\in E(E^n):
\{\varphi_0(e),\varphi_1(e)\}\cap I\neq\varnothing
\right\},
\]
the projection of $I$ in the $(n+1)$-st direction.

Let
\[
Y=\varphi_0(A)\cup\varphi_1(B).
\]
By the definition of $T$,
\[
T\cap(A\cup B)=\varnothing.
\]
Therefore, for every $i$, $0\leq i\leq p$,
\[
Y_i\cap Y=\varnothing.
\]
Moreover, by construction,
\begin{equation}
\pi_n(Y_i\cup Y)
=
C\cup\{e_1,\ldots,e_i\}.
\label{l1n3}
\end{equation}

We prove by induction that the 2-factor $G(Y_i\cup Y)$ has $k-i$
connected components for $0\leq i\leq p$.

Clearly,
\[
k(G(Y_0\cup Y))=k(G(A\cup B))=k.
\]

Suppose that the assertion holds for some $i$, $0\leq i<p$. Consider
the edges
\[
a=
(\varphi_0(x_0^{i+1}),\varphi_1(x_0^{i+1}))
\]
and
\[
b=
(\varphi_0(x_1^{i+1}),\varphi_1(x_1^{i+1})).
\]
By construction, $a,b\in Y_i$.

Suppose that $a$ and $b$ belong to the same connected component of
$G(Y_i\cup Y)$. Then the vertices $x_0^{i+1}$ and $x_1^{i+1}$ belong
to the same connected component of
\[
G\bigl(\pi_n(Y_i\cup Y)\bigr).
\]
It follows from \eqref{l1n2} and \eqref{l1n3} that
\[
k\bigl(G(C\cup\{e_1,\ldots,e_i\})\bigr)
=
k\bigl(G(C\cup\{e_1,\ldots,e_{i+1}\})\bigr),
\]
which contradicts \eqref{l1n1}.

Therefore, the edges $a$ and $b$ belong to different connected
components $\Psi_1$ and $\Psi_2$ of $G(Y_i\cup Y)$. Since
$G(Y_i\cup Y)$ is a 2-factor, $\Psi_1$ and $\Psi_2$ are simple cycles.
Removing the edges $a$ and $b$ and adding the edges
\[
\varphi_0(e_{i+1}),\qquad \varphi_1(e_{i+1})
\]
joins the two cycles $\Psi_1$ and $\Psi_2$ into a single cycle.
Consequently,
\[
k(G(Y_{i+1}\cup Y))
=
k(G(Y_i\cup Y))-1
=
k-i-1.
\]

Thus
\[
k(G(Y_p\cup Y))=k-p=m\leq l.
\]
Hence $Y_p$ is a perfect matching extending
$\varphi_0(A)\cup\varphi_1(B)$ to a 2-factor with at most $l$
connected components. The lemma is proved.
\end{proof}

Denote by $M_n^*$ the set of perfect matchings in $E^n$ that can be
extended to a Hamiltonian cycle in $E^n$, and denote by $H_n$ the set
of Hamiltonian cycles in $E^n$.

The case $l=1$ of Lemma~1 gives the following corollary.

\begin{corollary}
If $A\in M_n^*$ and $B\in M_n$, then
\[
\varphi_0(A)\cup\varphi_1(B)\in M_{n+1}^*.
\]
\end{corollary}

\begin{theorem}
For $n\geq3$,
\[
2^{n-1}
\left(
\ln n-1-\frac4n-\frac{\ln n}{2^{n/2-2}}
\right)
\leq
\ln|H_n|
\]
and
\[
\ln|H_n|
\leq
2^n
\left(
\ln n-1+
\frac{1}{2n}
\left(\ln(2\pi n)+\frac1{6n}\right)
\right).
\]
\end{theorem}

\begin{proof}
Since
\[
|H_n|\leq |F_n|,
\]
the upper bound for $|H_n|$ follows from Theorem~2.

Since
\[
|H_n|\geq \frac{|M_n^*|}{2},
\]
to prove the lower bound for $|H_n|$ it is sufficient to establish,
for $n\geq3$,
\begin{equation}
\ln|M_n^*|
\geq
2^{n-1}
\left(
\ln n-1-\frac4n-\frac{\ln n}{2^{n/2-2}}
\right).
\label{t1n1}
\end{equation}

By Theorem~1 and Corollary~1, for $n\geq3$ we have
\begin{equation}
\ln|M_n^*|
\geq
\sum_{i=2}^{n-1}\ln|M_i|
\geq
\sum_{i=2}^{n-1}2^{i-1}(\ln i-1).
\label{t1n2}
\end{equation}

We estimate the sum
\[
\sum_{i=2}^{n-1}2^{i-1}\ln i
\]
from above. By inequality~(6) and the identity
\[
\sum_{j=1}^{\infty}\frac{j}{2^j}=2,
\]
for $n\geq3$ we obtain
\[
\begin{aligned}
\sum_{i=2}^{n-1}2^{i-1}\ln i
&\geq
\sum_{i=\lceil n/2\rceil}^{n-1}
2^{i-1}
\left(
\ln n+\ln\left(1-\frac{n-i}{n}\right)
\right)\\
&\geq
\left(2^{n-1}-2^{\lceil n/2\rceil}\right)\ln n
-
2^{n-1}
\sum_{i=\lceil n/2\rceil}^{n-1}
2^{i-n}
\left(
\frac{n-i}{n}
+
\left(\frac{n-i}{n}\right)^2
\right)\\
&=
\left(2^{n-1}-2^{\lceil n/2\rceil}\right)\ln n
-
2^{n-1}
\sum_{j=1}^{n-\lceil n/2\rceil}
2^{-j}
\left(
\frac{j}{n}+\left(\frac{j}{n}\right)^2
\right)\\
&\geq
\left(2^{n-1}-2^{\lceil n/2\rceil}\right)\ln n
-
\frac{2^{n-1}}{n}
\sum_{j=1}^{\infty}\frac{2j}{2^j}\\
&\geq
2^{n-1}
\left(
\ln n-1-\frac4n-\frac{\ln n}{2^{n/2-2}}
\right).
\end{aligned}
\]
Combining this estimate with \eqref{t1n2} gives \eqref{t1n1}.
Therefore the lower bound in the theorem follows, and the theorem is
proved.
\end{proof}

\section{Extending a Perfect Matching to a Hamiltonian Cycle in $E^n$}

Denote by $M_n^k$, where $M_n^k\subset M_n$, the set of perfect
matchings in each of which the edges belong to at most $k$ different
directions. Clearly, $M_n=M_n^n$. The following statement holds.

\begin{theorem}
For every integer $k>0$, there exists an integer $n_0\geq k$ such that,
for all integers $n\geq n_0$,
\[
M_n^k\subset M_n^*.
\]
\end{theorem}

\begin{proof}
Since the length of the shortest simple cycle in $E^k$ is $4$, any two
disjoint perfect matchings $A$ and $B$ in $M_k$ form a 2-factor
$G(A\cup B)$ having at most $2^{k-2}$ connected components.

We prove by induction that, for every $i$ satisfying
\[
0\leq i<2^{k-2}-1,
\]
every perfect matching in $M_{k+i}^k$ can be extended to a 2-factor
having at most $2^{k-2}-i$ connected components.

Indeed, suppose that the induction hypothesis holds for some
\[
0\leq i<2^{k-2}-2.
\]
Consider a perfect matching
$A\in M_{k+i+1}^k$. Without loss of generality, we may assume that
$A$ contains no edges in the $(k+i+1)$-st direction. Consequently,
$A$ can be represented as
\[
A=\varphi_0(A_0)\cup\varphi_1(A_1),
\]
where $A_0,A_1\in M_{k+i}^k$.

By the induction hypothesis, there exists $B\in M_{k+i}$ such that
$A_0\cap B=\varnothing$ and the 2-factor $G(A_0\cup B)$ has $m$
connected components, where
\begin{equation}
m\leq 2^{k-2}-i.
\label{t2n1}
\end{equation}

Since $E^{k+i}$ is connected, there exists an edge
\[
e=(x_1,x_2)\in E(E^{k+i})
\]
such that $x_1$ and $x_2$ belong to different connected components of
$G(A_0\cup B)$.

Remove from the 2-factor $G(A_0\cup B)$ the two edges belonging to $B$
that are incident with $x_1$ and $x_2$, and add the edge $e$.
In each of the remaining cycles of the 2-factor $G(A_0\cup B)$,
remove one arbitrary edge belonging to $B$.

As a result, we obtain a graph
\[
G=G(A_0\cup B'),
\]
where the matching $B'$ is contained in $B\cup\{e\}$. The graph
$G(A_0\cup B')$ consists of $m-1$ simple paths.

It follows from \eqref{t2n1} and Lemma~1 that the perfect matching $A$
can be extended to a 2-factor having at most
\[
2^{k-2}-i-1
\]
connected components.

Thus, when
\[
i=2^{k-2}-1,
\]
an arbitrary perfect matching in $M_{k+i}^k$ can be extended to a
Hamiltonian cycle. Hence
\[
M_{n_0}^k\subset M_{n_0}^*,
\qquad
n_0=k+2^{k-2}-1.
\]
Moreover, Corollary~1 implies that
\[
M_n^k\subset M_n^*
\]
for all $n\geq n_0$.

This proves Theorem~4.
\end{proof}

Thus, every perfect matching whose edges belong to at most $k$
directions can be extended to a Hamiltonian cycle in $E^n$ for
$n\geq n_0(k)$.

\end{document}